\documentclass[11pt]{article}
\usepackage{amssymb, amsthm, amsmath, amscd}

\usepackage[monochrome]{color}

\usepackage{ulem}

\newtheorem{thm}{Theorem}[section]
\newtheorem{lem}[thm]{Lemma}
\newtheorem{prop}[thm]{Proposition}
\newtheorem{cor}[thm]{Corollary}

\theoremstyle{definition}\newtheorem{df}[thm]{Definition}
\theoremstyle{definition}\newtheorem{rem}[thm]{Remark}
\theoremstyle{definition}

\renewcommand{\phi}{\varphi}

\newcommand{\N}{\mathbb{N}}
\newcommand{\Z}{\mathbb{Z}}
\newcommand{\Q}{\mathbb{Q}}
\newcommand{\R}{\mathbb{R}}

\newcommand{\morp}{contractive completely positive linear map}

\newcommand{\hm}{homomorphism}
\newcommand{\dt}{\delta}
\newcommand{\ep}{\epsilon}

\newcommand{\andeqn}{\,\,\,{\rm and}\,\,\,}
\newcommand{\rforal}{\,\,\,{\rm for\,\,\,all}\,\,\,}
\newcommand{\CA}{$C^*$-algebra}
\newcommand{\SCA}{$C^*$-subalgebra}

\newcommand{\beq}{\begin{eqnarray}}
\newcommand{\eneq}{\end{eqnarray}}
\newcommand{\tforal}{\,\,\,\text{for\,\,\,all}\,\,\,}

\title{ Tensor Products of Classifiable $C^*$-algebras}
\author{ Huaxin Lin and Wei Sun\thanks{The corresponding author}
 }
\date{}

\begin{document}

\maketitle
\begin{abstract}

Let ${\cal A}_1$ be the class of all unital separable simple \CA s $A$ such that
$A\otimes U$ has tracial rank {no more than} one for all UHF-algebra {$U$} of infinite type. It has been shown
that {all } amenable ${\cal Z}$-stable \CA s in ${\cal A}_1$ which satisfy the Universal Coefficient Theorem can be classified
up to isomorphism by the Elliott \textcolor{blue}{invariant}. {In this note}, {we} show that $A\in {\cal A}_1$ if and only if
$A\otimes B$ has tracial rank {no more than} one for {some} unital simple infinite dimensional AF-algebra $B.$
In fact, we show that $A\in {\cal A}_1$ if and only if $A\otimes B\in {\cal A}_1$ for some
unital simple AH-algebra $B.$ We actually prove a more general result. Other results regarding the tensor products of \CA s in ${\cal A}_1$ are also obtained.

\end{abstract}

\section{Introduction}

The Elliott program of classification of amenable \CA s is to classify separable amenable \CA s up to {}{isomorphism}
by {their} $K$-theoretic data known as the Elliott {invariant}.  It is {a} very successful program.  Two important classes of unital separable simple \CA s, the class of amenable separable purely infinite simple \CA s satisfying the Universal Coefficient
Theorem (UCT) and {the class of} unital simple AH-algebras with no dimension growth{}{,} are classified by their Elliott {invariant} (see
\cite{KP}, \cite{EG-0}, \cite{EGL} and \cite{G1} among {many articles in the literature}). There {has been other significant progress} in the Elliott program.  Related to this note, it has been shown that unital separable amenable
simple \CA s with tracial rank at most one {}{and} satisfy the UCT are classifiable by the Elliott {invariant}. {}{In fact}, they are isomorphic to unital simple AH-algebras with no dimension growth. More recently,
with a remarkable   method developed by W. Winter (\cite{W-local}), {the notion of rational tracial rank at most one was introduced (a unital separable simple \CA\, $A$ is said to have rational tracial rank at most one if $A\otimes U$ has tracial rank at most one
for every  UHF-algebra $U$ of infinite type), and} it was shown in \cite{Lin-AsymUnitEquiv}
that unital separable amenable simple ${\mathcal Z}$-stable \CA s which satisfy the UCT and {}{have rational} tracial rank at most one are also classifiable by the Elliott {invariant} (see also \cite{W-local}, \cite{Lin-Appendix} and \cite{LN-lifting}).  This class is significantly larger than the class of all
unital simple AH-algebras with no dimension growth.
  Denote by ${\mathcal A}_1$ the class of all unital separable simple \CA s which {}{have } rational tracial rank at most one.
A special unital separable simple \CA\, in ${\mathcal A}_1$ which does not have finite tracial rank is the Jiang-Su algebra ${\mathcal Z}.$ The range of the Elliott {}{invariant} for {}{\CA s of rational} tracial rank at most one has been characterized and computed
(see \cite{LN-range}).
This class of \CA s {includes} \CA s whose ordered $K_0$-groups may not have the Riesz interpolation property.
The verification that a  particular unital simple \CA\, is in the class ${\mathcal A}_1$ was slightly eased when
it was proved in \cite{LN-range} that, $A\in {\mathcal A}_1$ if and only if $A\otimes U$ has tracial rank at most one for some UHF-algebra $U$ of infinite type ({instead} {}{of} for all UHF-algebras of infinite type).
Suppose {}{that} $A$ is a unital separable simple \CA\, such that $A\otimes B$ has tracial rank at most one for some unital {}{simple infinite dimensional} AF-algebra $B.$ Does it follow that $A\in {\mathcal A}_1?$  We will answer this question
affirmatively in this short note. In fact, we will show that if $A\otimes B$ has tracial rank at most one for some unital infinite dimensional separable
simple \CA\, $B$ with tracial rank at most one then $A\in {\mathcal A}_1.$
This may provide a better way to determine which \CA s are in ${\mathcal A}_1.$  

{Denote by ${\mathcal N}$ the class of all unital separable amenable $C\sp*$-algebras which satisfy the Universal Coefficient Theorem.}
For the {}{purpose of classification}, we also consider  ${\mathcal A}_1 {\cap {\mathcal N}}$, the class of all unital separable simple amenable \CA s which {}{have rational tracial rank} at most one {and satisfy} the UCT.
We will show that if $A$ and $B$ are {both} in ${\mathcal A}_1 {\cap {\mathcal N}}$, then $A\otimes B$ {}{is also in} ${\mathcal A}_1 {\cap {\mathcal N}}$.
{}{Assume} that $A\in {\mathcal A}_1 {\cap {\mathcal N}}$ and $B$ is \textcolor{blue}{a simple \CA} \ with tracial rank at most one \textcolor{blue}{and satisfies the UCT}.
{}{From the fact above}, $A\otimes B$ is also in ${\mathcal A}_1 {\cap {\mathcal N}}$. One {}{might} ask
whether $A\otimes B$ has tracial rank at most one. We will {also} give an affirmative answer to this question.  

Most of results are in a more general setting which may provide an opportunity for the future applications.
In fact, with a much more recent classification result in \cite{GLN}, we expect some of the results presented 
in this note can be used to ease some technical constrains.  In fact, for example, \textcolor{blue}{for a unital simple separable \CA\, $A$,
it is much \textcolor{blue}{more} delightful to work \textcolor{blue}{with} $A\otimes Q$ than $A\otimes U,$} since 
$K_i(A\otimes Q)$ ($i=0,1$) is torsion free and divisible, while \textcolor{blue}{$K_i(A\otimes U)$} could have torsion in general. 
Some applications of results in \textcolor{blue}{this} short note can be found in \cite{GLN}.

\section{Preliminaries}
\begin{df}
{Let} $A$ be a \CA.
Let $\mathcal{F}$ and $\mathcal{G}$ be two subsets {of $A$}. Let $\ep > 0$. We say that $\mathcal{F} \subset_{\ep} \mathcal{G}$
if for each $x \in \mathcal{F}$, there exists $y \in \mathcal{G}$, such that $\| x - y \| < \ep$.

By $A_+$, we mean the positive cone of all positive elements in $A$.

If $a, b\in A_+$, we write
$a\lesssim b$ if there \textcolor{blue}{is a sequence $\{ x_n \}$ in $A$} such that  $\lim_{n\to\infty} \|x_n^*bx_n- a\|=0.$
We say two positive elements $x$ and $y$ are Cuntz equivalent and write \textcolor{blue}{it as} $x\sim y,$  if $x\lesssim y$ and 
$y\lesssim x.$ 

Let $A$ be a unital stably finite simple \CA. 
Denote by $T(A)$ the tracial state space of $A.$ Define
$d_\tau(a)=\lim_{n\to\infty} \tau(a^{1/n})$ for all $a\in A_+$ and $\tau\in T(A).$ 
$A$ is said to have strict comparison property for positive elements if 
for any pair $a, b\in A_+\setminus\{0\},$ $d_\tau(a)<d_{\tau}(b)$ for all $\tau\in T(A)$ implies that 
$a\lesssim b.$

{Let ${\mathcal F} \subset A$ be a finite subset and let $p \in A$ be a projection. We use $p {\mathcal F} p$ to
denote $\{ p x p \colon x \in {\mathcal F} \}$.


}

\end{df}

\begin{df}\label{DI0}
Let ${\cal B}$ be a family of unital \CA s. 
We say a unital simple separable \CA\, $A$ is tracially approximated by \SCA s in ${\cal B}$ and write \textcolor{blue}{it as} 
$A\in TA{\cal C},$ if the following holds:
For any $\ep>0,$ any finite subset ${\mathcal F}\subset A$ and any $a\in A_+\setminus\{0\},$ there {}{exist
a projection $p\in A$ and a \SCA\, $B\subset A$} with {$B\in {\cal B}$} and $1_B=p$ such that
\begin{align}\label{DI0-1}
& \|px-xp\| < \ep\tforal x\in {\mathcal F},\\
& p {\mathcal F} p  \subset_{\ep} B\andeqn\\
& 1-p  \lesssim a.
\end{align}

Let ${\cal B}={\cal I}_1$ be the family of \CA s of the form $C([0,1], F),$ where $F$ is a unital finite dimensional \CA .
Then we write $TR(A)\le 1$ if $A\in TA{\cal I}_1.$ 

Note  that, in the original definition {3.1} of  \cite{Lin-TR-LMS}, {${\mathcal I}_1$} 
is replaced by the class of
all finite direct sums of \CA s of the form $M_n(C({}{X_n})),$ where {}{each $X_n$ is a finite CW {complex} with dimension
{one}}.   But those definitions are equivalent. Please see Theorem 6.13 and Theorem 7.1 of \cite{Lin-TR-LMS} for more details on such equivalence.  {In the definition above, if we replace
${\cal B}$  by {${\mathcal I}_0$}, the class of finite dimensional \CA s, then we say that $A$ has tracial rank zero (see Theorem 7.1 of \cite{Lin-TR-LMS}}).  If $A$ has tracial rank at most one, we {denote it by} $TR(A)\le 1$. If $A$ has tracial rank zero, {}{we denote it by} $TR(A)=0$. {For more details, please see  \cite{Lin-TR-LMS}.}

\end{df}

\noindent \textbf{Notations:}
Let $A$ be a unital \CA.  {For each $n \in \N$, there is an embedding of $M_n(A)$ into $M_{n + 1}(A)$ defined by $a \mapsto \begin{pmatrix} a & 0 \\ 0 & 0  \end{pmatrix}$}. Denote by $M_{\infty}(A)$ {}{the {algebraic inductive limit of $M_1(A) \rightarrow M_2(A) \rightarrow \cdots$, whose connecting maps are just the embeddings above}}.
Suppose that \textcolor{blue}{$T(A)\not= 0.$} 
 {For any $p \in M_{\infty}(A)$ and for any $\tau \in T(A)$, we may assume that $p \in M_{n}(A)$ for certain $n.$  By identifying $M_n(A)$ with $A \otimes M_n(\mathbb{C})$,
 we define $\tau(p)$ to be} $(\tau\otimes \mbox{Tr})(p),$ where $\mbox{Tr}$ is the standard {matrix} trace {}{(not normalized)}
on $M_n(\mathbb{C}).$ Note that the value $(\tau\otimes \mbox{Tr})(p)$ is independent
of the choice of $n.$

Denote by ${\mathcal N}$ the class of all unital separable amenable $C\sp*$-algebras which satisfy the Universal Coefficient Theorem.

Denote by $Q$ the UHF-algebra with $(K_0(Q), K_0(Q)_+, [1_Q])=(\Q, \Q_+, 1).$

{}{Use} ${\mathcal A}_0$ to denote the class of all unital separable simple  $C\sp*$-algebras $A$ for which
$TR(A \otimes M_{\mathfrak{p}}) = 0$ for all supernatural numbers $\mathfrak{p}$ of infinite type

Use ${\mathcal A}_1$ to denote the class of all unital separable simple  $C\sp*$-algebras $A$ for which
$TR(A \otimes M_{\mathfrak{p}}) {\leq} 1$ for all supernatural numbers $\mathfrak{p}$ of infinite type.

By the above defined notations, ${\mathcal A}_0 \cap {\mathcal N}$ is the class of all \CA s {}{which are amenable and are in ${\mathcal A}_0$}, and ${\mathcal A}_1 \cap {\mathcal N}$ is the class of all \CA s {}{which are amenable and are in ${\mathcal A}_1$}.


\vspace{2mm}


\begin{df}
Let $\ep>0.$ 
Define
$${ f_{\ep}(t) = \left\{ \begin{array}{lll} 1 & & t \geq 2 \ep \\ (1 / \ep) t - 1 & & \ep < t < 2 \ep \\ 0 & &{0 \leq} \ t \leq {\ep} \end{array} \right.} {.} $$
Then $f$ is a continuous function on $[0, \infty).$
\end{df}

\section{Tensor with AF-algebras}

\begin{df}\label{DC}
Throughout this section and the next, 
let ${\cal C}$ be a class of unital separable amenable \CA s which satisfy the following properties:
(1)
 Every finite dimensional \CA s is in ${\cal C};$ (2) \textcolor{blue}{If} $A\in {\cal C},$ then $A\otimes F\in {\cal C},$ for every 
finite dimensional \CA\, $F;$ (3) \textcolor{blue}{Every} \CA\, in ${\cal C}$ is weakly semiprojective; 
(4) \textcolor{blue}{Every} unital hereditary \SCA\, of \CA s in ${\cal C}$ is in ${\cal C};$ (5) Suppose that 
$A\in {\cal C}$ and $I\subset A$ is a closed ideal. Then, for any finite subset ${\cal F}\subset A/I$ and any $\ep>0,$ 
there exists a \SCA\, $B\subset A/I$ such that $B\in {\cal C}$ and
${\rm dist}(x, B)<\ep$ for all $x\in {\cal F}.$

It is easy to verify that the class ${\cal I}_1$ satisfies (1) \textcolor{blue}{to} (5). 

Let $F_1,$ $F_2$ be two finite dimensional \CA s, \textcolor{blue}{and let} $\phi_1,\, \phi_2: F_1\to F_2$ be two \hm s.
Define the mapping torus 
$$
A=A(F_1, F_2, \phi_1, \phi_2)=\{(f,a)\in C([0,1], F_2)\oplus F_1: f(0)=\phi_1(a)\andeqn f(1)=\phi_2(a)\}.
$$
Let ${\cal C}'$ be the class  consisting of all such mapping tori and all finite dimensional \CA s. 
It is obvious that ${\cal C}'$ satisfy the property (1) and (2) above. 
It is proved in \cite{ELP} that all \CA s in ${\cal C}'$ are semiprojective (property (3)).  It is proved in \cite{GLN} that
the class also satisfies property \textcolor{blue}{(4) and (5)}.
\end{df}

We begin with the following:

\begin{prop}\label{L1}
Let $A$ be a unital separable simple {infinite dimensional} \CA\, in $TA{\cal C}.$ Then, for any simple AF-algebra
$B$ ($B$ could be finite dimensional), $A\otimes B\in TA{\cal C}.$ 	
\end{prop}

  \begin{proof}  The case that $B$ is finite dimensional follows from the property (1) of \CA s in ${\cal C}.$ 
 
  {Now we  assume $B$ is infinite dimensional.} 
   It is easy to see that $A\otimes B$ is a unital simple \CA. Note that $B$ is approximately divisible (see \cite{BKR} for the definition).
  By Theorem 1.4 of \cite{BKR}
   $A\otimes B$ has  the strict comparison property for positive elements. 
  Let ${\cal F}\subset A\otimes B$ be a finite subset, $\ep>0$ and $c\in (A\otimes B)_+\setminus\{0\}.$ 
 Since $A$  is a unital  infinite dimensional simple \CA,
 it is non-elementary.  It is easy to find, for any integer $n\ge 1,$ $n$ non-zero mutually orthogonal 
 and Cuntz equivalent positive elements. By the strict comparison, 
 one obtains  a non-zero element $a_0\in A_+$ such that $a_0\otimes {1_B} \lesssim c.$ 
 
  To prove that $A\otimes B$ is in $TA{\cal C},$ we may assume, without  loss of generality, that 
  ${\cal F}=\{a\otimes b: a\in {\cal F}_1\andeqn b\in {\cal F}_2\},$ where ${\cal F}_1$ and ${\cal F}_2$ are 
  finite subsets in $A$ and $B,$ respectively. Since $B$ is AF, \textcolor{blue}{we may} further assume 
  that ${\cal F}_2\subset F,$ where $F$ is a  unital finite dimensional \SCA\, of $B.$ 
  Moreover, to simplify notation further, without loss of generality, 
  we may also assume that $\|a\|,\, \|b\|\le 1$  {for all $a \in {\cal F}_1$ and $b \in {\cal F}_2$}.
  
  Since $A\in TA{\cal C},$ there exists a projection $p_1\in A$  and a \SCA\, $C_0\in {\cal C}$ of $A$ 
  with $1_{C_0}=p_1$ such that 
  \beq\label{Tred-1}
  \|ap_1-p_1a\| & < & \ep/2\rforal a\in {\cal F}_1,\\
  {\rm dist}(p_1ap_1, C_0) & < & \ep/2\rforal a\in {\cal F}_1\andeqn\\
  1-p_1 &\lesssim & a_0.
  \eneq
  Define $C_1=C_0\otimes F$ and $p=p_1\otimes 1_B.$  Then $C_1\in {\cal C}$ and $1_{C_1}=p.$
  \textcolor{blue}{It follows that}
  \beq\label{Tred-2}
  \|xp-px\| & < & \ep\rforal x\in {\cal F}\\
  {\rm dist}(pxp, {C_1}) & < & \ep\rforal x\in {\cal F}\andeqn\\
  1-p &\lesssim &  a_0\otimes 1_B\lesssim c.
  \eneq

  \end{proof}

\begin{rem}
\textcolor{blue}{If $A$ is finite dimensional, Proposition \ref{L1} still holds.}
\end{rem}

\begin{thm}\label{TN1}
Let $A$ be a unital separable simple \CA. 
Suppose that 
$A\otimes U\in TA{\cal C}$ for some infinite dimensional  UHF-algebra $U.$ 
Then $A\otimes B\in TA{\cal C}$ for any unital infinite dimensional simple AF-algebra. 
\end{thm}

\begin{proof}
Suppose that $A\otimes \textcolor{blue}{U} \in {TA{\cal C}.}$  
Let $B$ be a unital infinite dimensional simple AF-algebra.

Fix $\ep>0,$ a finite subset ${\cal F}\subset A\otimes B$ and $a\in (A\otimes B)_+\setminus\{0\}.$

\textcolor{blue}{As $B$ is approximately divisible, so is $A \otimes B$.} It follows from \textcolor{blue}{Theorem 1.4 (a) of} \cite{BKR} that $A\otimes B$ is either purely infinite or 
%
has the strict comparison property for positive elements.
In either \textcolor{blue}{case}, there is a non-zero element $a_0\in \textcolor{blue}{1_A} \otimes B$ such that 
$a_0\lesssim a$ in $A\otimes B.$ 
\textcolor{blue}{As $A \otimes B$ is simple,} there is an integer $N_0\ge 1$ such that 
\beq\label{27-1}
1_{\textcolor{blue}{A \otimes B}} \lesssim N_0[a_0].
\eneq

We write $B=\lim_{n\to\infty}(B_n, \psi_n),$ where each $B_n$ is a finite dimensional \CA\, and 
$\psi_n: B_n\to B_{n+1}$ is a unital embedding.   If $n>m,$ put 
$\psi_{m,n}=\textcolor{blue}{\psi_{n-1}} \circ\cdots \psi_m: B_m\to B_n.$ We will also use 
$\psi_{n, \infty}: B_n\to B$ for the \textcolor{blue}{unital} embedding induced by the inductive limit. 
Write 
$$
B_n=M_{R(n,1)}\oplus M_{R(n,2)}\oplus \cdots M_{R(n, k(n))}.
$$
\textcolor{blue}{According to Proposition 2.2 and Lemma 2.3 (b) of \cite{Rordam2},} to simplify notation, without loss of generality, by replacing 
$a_0$ \textcolor{blue}{with} a smaller (in Cuntz relation) element, we may assume that $a_0\in B_n$ for some 
large $n.$ Moreover, we may assume that $a_0=a_{1,n}\oplus a_{2,n}\oplus\cdots a_{k(n),n},$ where 
$a_{i,n}\in B_{R(n,i)},$ $i=1,2,...,k(n).$ 
Since $B$ is simple, we may assume that $R(n,j)>4N_0$ for all $j$ and all $n.$ 
It follows from (\ref{27-1}) that we may assume 
that the range projection of $a_{j,n}$ has rank at least two. 
\textcolor{blue}{Then} we may write $a_{j,n}\ge a_{j,n}^{(0)}+a_{j,n}^{(1)},$ where 
$a_{j,n}^{(i)}$ has exactly rank one range projection, and $a_{j,n}^{(0)}$ and $a_{j,n}^{(1)}$ are mutually orthogonal.
Thus 
$$
a_0\ge a_0^{(i)}=a_{1,n}^{(i)}\oplus a_{2,n}^{(i)}\oplus\cdots \oplus a_{k(n),n}^{(i)},\,\,\,i=0,1.
$$
By choosing possibly smaller $a_0$ (in the Cuntz relation), we may assume 
that $a_{j,n}^{(i)}$ is a rank one projection for each $j$ and $n,$ $i=0,1.$

 By changing notation, 
without loss of generality, we may further assume that
${\cal F}\subset A\otimes B_1$ and $a_0, a_0^{(0)}, a_0^{(1)}\in {B_1}.$ 
Define $\pi_{j}: B_1 \to M_{R(1, j)}$ to be the \textcolor{blue}{canonical} projection to the $j$-th summand, $j=1,2,..., \textcolor{blue}{k(1)},$ $n=2,3,....$ 
Put ${\cal F}_j= \textcolor{blue}{\pi_{j}} ({\cal F}),$ $j=1,2,...,k(1).$

\textcolor{blue}{For each}  $A\otimes M_{\textcolor{blue}{R(1, i)}}\otimes U \in  TA{\cal C},$ there exists a projection $p_i\in A\otimes M_{R(1,i)}\otimes U$
 and 
a \SCA\, $\textcolor{blue}{D_{0,i}} \in {\cal C}$ with $\textcolor{blue}{ 1_{D_{0,i}} }= p_i$ such that
\beq\label{NT-1}
\|[p_i,\, x]\| & < & \ep/8\rforal x\in {\cal F}_i,\\
{\rm dist}(p_ixp_i, \textcolor{blue}{ D_{0,i} } ) & < & \ep/8\rforal x\in {\cal F}_i\andeqn\\\label{NT-1+}
1_{A\otimes M_{R(1,i)} \textcolor{blue}{ \otimes U } }-p_i & \lesssim & a_{i,1}^{(0)}.
\eneq	
Let ${\cal G}_i\subset B_{0,i}$ be a finite subset such that, for every $x\in {\cal F}_i,$ there exists $x'\in {\cal G}_i$ such 
that $\|p_ixp_i-x'\|<\ep/\textcolor{blue}{16}.$  We may also assume that $1_{  \textcolor{blue}{D_{0,i}}  }\in {\cal G}_i.$ 

Write $U=\overline{\cup_{n=1}^{\infty} M_{r(n)}},$ where 
$\lim_{n\to\infty} r(n)=\infty$ and $M_{r(n)}\subset M_{r(n+1)}$ unitally. 
\textcolor{blue}{Since each \textcolor{blue}{$D_{0,i}$} is weakly semiprojective, we can choose $n_0$ large enough,} \textcolor{blue}{such that  for each $i = 1, 2, \cdots, k(1)$}, there exists a unital \hm\, $\phi_i: \textcolor{blue}{ D_{0,i} } \to A\otimes M_{R(1,i)}\otimes M_{r(n_0)}$, \textcolor{blue}{satisfying}
\beq\label{NT-2}
\|\phi_i(x')-x'\|<\ep/8\rforal x'\in {\cal G}_i.
\eneq
Without loss of generality (by replacing $\phi_i$ \textcolor{blue}{with} ${\rm Ad}\, u_i\circ \phi_i$ for some unitary $u_i$ \textcolor{blue}{in $A\otimes M_{R(1,i)}\otimes M_{r(n_0)}$} if necessary), we may assume that $\phi_i(1_{B_{0,i}})=p_i.$
It follows from the property (5)  of \CA s in ${\cal C}$ that there exists a unital \CA\, \textcolor{blue}{$D_i\subset \phi_i(D_{0,i})$} such that
\textcolor{blue}{$D_i\in {\cal C}$}, \textcolor{blue}{$1_{D_i}=\phi_i(D_{0,i})$} \textcolor{blue}{and}
\beq\label{NT-3}
{\rm dist}(\phi_i(x'), \textcolor{blue}{D_i} )<\ep/\textcolor{blue}{16} \rforal x'\in {\cal G}_i.
\eneq
Note that \textcolor{blue}{$1_{D_i}=\phi_i(1_{D_{0,i}})=p_i.$} 
Thus 
\beq\label{NT-3}
{\rm dist}(p_ixp_i, \textcolor{blue}{D_i} )<\ep/4\rforal x\in {\cal F}_i.
\eneq

Denote by $\imath_{0,i}: M_{R(1,i)}\to M_{R(1,i)r(n_0)}$ \textcolor{blue}{the map defined as} $\imath_{0,i}(x)=x\otimes 1_{M_{r(n_0)}},$
$i=1,2,...,k(1),$ and define $\imath_0: A\otimes B_1\to A\otimes B_1\otimes M_{r(n_0)}$ 
by $\imath_0(x)=x\otimes 1_{M_{r(n_0)}}$ for all $x\in A\otimes B_1.$ 

Since $B$ is a unital simple AF-algbera, we may assume that $\psi_{1, n_1}: B_1\to B_{n_1}$ has 
multiplicities at least $N\ge 1$ \textcolor{blue}{for each simple  summand of $B_1$}, such that
\beq\label{NT-4}
{2r(n_0) \left( \sum_{j=1}^{\textcolor{blue}{ k(1)} }R(1,j) \right)^2\over{N}}<1.
\eneq
Put $\Psi_{i,j}=\pi_{n_1,j}\circ \left(  \psi_{1, n_1}|_{M_{R(1,i)}} \right): M_{R(1,i)}\to M_{R(n_1,j)},$ \textcolor{blue}{where $\pi_{n_1,j}$ is the canonical projection to the $j$-th summand of $B_{n_1}$.} The assumption  on the multiplicity \textcolor{blue}{implies} 
that 
\textcolor{blue}{ $\Psi_{i,j}(1_{ M_{R(1,i)} })=1_{ M_{m(i,j)} } \in M_{R(n_1,j)}$ } with $m(i,j)\ge N,$ $i=1,2,...,k(1)$ and 
$j=1,2,...,k(n_1).$ It follows that 
$R(1,i)|m(i,j),$ $i=1,2,...,k(1)$ and $j=1,2,...,k(n_1).$ 
Note that \textcolor{blue}{$1_{ M_{R(n_1,j)} }=\bigoplus_{i=1}^{k(1)} \Psi_{i,j}(1_{ M_{R(1,i)} }),$ $j=1,2,...,k(n_1).$} 
Write 
\beq\label{NT-5}
m(i,j)=l(i,j)r(n_0)R(1,i)+s_{i,j},
\eneq
where $l(i,j)\ge 1$ and $r(n_0)R(1,i)>s_{i,j}\ge 0$ are integers. It follows that  
\beq\label{NT-5+}
\sum_{i=1}^{k(1)}{s_{i,j}\over{m(i,j)}} \textcolor{blue}{ < \sum_{i=1}^{k(1)}{r(n_0)R(1,i)\over N } } < \sum_{i=1}^{k(1)}{r(n_0)R(1,i)\over{2r(n_0)(\sum_{i=1}^{k(1)}R(1,i))^2}}<{1\over{2\sum_{i=1}^{k(1)}R(1,i))}}.
\eneq
Since $R(1,i)|m(i,j),$ we may write $s_{i,j}=s_{i,j}^{  \textcolor{blue}{(r)} }R(1,i),$ $i=1,2,...,k(1)$ and $j=1,2,...,k(n_1).$ 
Define $\rho_{i,j}: M_{R(1,i)}\to M_{s_{i,j}}$ by $x\to x\otimes \textcolor{blue}{  1_{M_{s_{i,j}^{(r)} }} }.$
 Note also that
$\sum_{i=1}^{k(1)}m(i,j)=R(n_1,j),$ $j=1,2,...,k(n_1).$

It follows from (\ref{NT-5+}) that 
\beq\label{NT-5++}
\bigoplus_{i=1}^{k(1)}\rho_{i,j}(1_{M_{R(1,i)}}) \, \lesssim \, \psi_{1, n_1}(a_{j,1}^{(1)}).
\eneq	
Let $\imath_{1,i,j}: M_{r(n_0)R(1,i)}\to M_{l(i,j)r(n_0)R(1,i)}$ be the embedding defined 
by $a\mapsto a\otimes 1_{M_{l(i,j)}}.$  Let  $\imath_{2,i,j}: M_{l(i,j)r(n_0)R(1,i)}\to \Psi_{i,j}(M_{R(1,i)})$ be defined 
by the embedding which sends rank one projection to rank one projection.
Put $\imath_{3,i,j}=\imath_{2,i,j}\circ \imath_{1,i,j}.$  Define $\imath_{4,i,j}: A\otimes M_{R(1,i)}\otimes M_{r(n_0)}\to A\otimes M_{R(n_1,j)}$
\textcolor{blue}{by} $\imath_{4,i,j}(a\otimes b)=a\otimes \imath_{3,i,j}(b)$ for all 
$a\in A$ and $b\in M_{r(n_0)R(1,i)}.$ 
Note that 
$$
\left( \oplus_{j=1}^{k(n_1)} \imath_{3,i,j}\circ \imath_{0,i} \right) \bigoplus \left( \oplus_{j=1}^{k(n_1)}\rho_{i,j} \right) \, = \, \bigoplus_{j=1}^{k(n_1)}\Psi_{i,j} \ = \ \psi_{1,n_1}|_{M_{R(1,i)}}
$$ 
and 
$$ \bigoplus_{i=1}^{k(1)} \left( \left( \oplus_{j=1}^{k(n_1)}\imath_{3,i,j}\circ \imath_{0,i}\right) \oplus \  \left( \oplus_{j=1}^{k(n_1)}\rho_{i,j}\right) \right) \ = \ \psi_{1, n_1}.$$

Define $\imath: A\otimes B_{n_1}\to A\otimes B$ to  be the map given by $a\otimes b\mapsto a\otimes \psi_{1, \infty}(b).$ 

Put $C_1=\bigoplus_{i=1}^{k(1)} \imath \left( \oplus_{j=1}^{k(n_1)}\imath_{4,i,j}( \textcolor{blue}{D_i} ) \right).$ Then 
$C_1\in {\cal C}$ and $p=1_{C_1}$ has the form \textcolor{blue}{$\imath \left( \oplus_{i=1}^{k(1)}p_i\otimes p_i' \right),$} where 
$p_i\in A\otimes M_{R(1,i)}\otimes M_{r(n_0)}$  and $p_i'=\oplus_{j=1}^{k(n_1)}1_{M_{l(i,j)}}.$ A fact we use here 
is 
\beq\label{NT-6}
\|xp-px\|=\|x(p_i\otimes p_i')-(p_i\otimes p_i')x\|=\|xp_i-p_ix\|<\ep / 4
\eneq
for all $x\in \textcolor{blue}{\cal F},$ since $x=a\otimes b,$ where $a\in A$ and $b\in B_1.$ 
We also have 
\beq\label{NT-7}
{\rm dist}(pxp, C_1)<\ep/4\rforal x\in {\cal F}.
\eneq
By (\ref{NT-1+}) and  (\ref{NT-5++}),
\beq\label{NT-8}
1-p\le \sum_{i=1}^{k(1)}(1-p_i)+\sum_{j=1}^{k(n_1)}\rho_{i,j}(1_{M_{R(1,i)}})\lesssim a_0^{(0)}+a_0^{(1)}\lesssim a,
\eneq
where we identify $a_0^{(i)}$ with $\psi_{1,\infty}(a_0^{(i)}).$
Therefore $1-p\lesssim a. $
This implies that $A\otimes B\in {TA{\cal C}}.$ 

\end{proof}

\begin{prop}\label{Lher}
Let $A$ be a unital separable simple \CA\, in $TA{\cal C}$ and let $p\in A$ be a non-zero projection.
Then $pAp\in TA{\cal C}.$
\end{prop}

\begin{proof}
 Let $1/4>\ep>0.$ {Let} ${\cal F}\subset pAp$ be a finite subset and let $a\in {(pAp)}_+\setminus \{0\}.$
Without loss of generality, we may assume that $p\in {\cal F}$ and $\|x\|\le 1$ for all $x\in {\cal F}.$  
Since $A$ is in $TA{\cal C},$ there is a projection $e\in A$ and a \SCA\, $C_0\in {\cal C}$ of $A$ with $1_{C_0}=e$ such that
{\color{blue}{
\beq\label{Lher-1}
\|ex-xe\| & < &\ep/64\rforal x\in {\cal F},\\
{\rm dist}(exe, C_0) & < & \ep/64\rforal x\in {\cal F}\andeqn\\\label{Lher-1+}
1-e & \lesssim & a.
\eneq
}}
We have $\|ep-pe\|<\ep/64.$ One computes that there is a projection $e_0\le p$  and $e_0'\in C_0$ such that
$\|e_0-pep\|<\ep/ {32},$ $\|e_0'-epe\|<\ep/{32}$ and 
$\|e_0-e_0'\|<\ep/{8}.$  \textcolor{blue}{Then there is} a unitary $u\in A$ such that $\|u-1\|<\ep/4$ such that
$u^*e_0'u=e_0.$ Define $C_1=u^*(e_0'C_0e_0')u.$ By the property (4) {as in the definition of $\mathcal{C},$} $e_0'C_0e_0'\in {\cal C}.$
Therefore $C_1\in {\cal C}.$ 
Note that $px=xp=x$ for all $x\in {\cal F}.$  We then estimate that ({with} $\|x\|\le 1$ for $x\in {\cal F}$ as assumed), \textcolor{blue}{$\rforal x\in {\cal F}$,} 
\beq\label{Lher-3}
\|e_0x-xe_0\| &\le & \|e_0x-pepx\|+\|pepx-xe_0\|\\
&<& \ep/16+\|pepx-pxep\|+\|pxep-xe_0\|\\
&<& \ep/16+\ep/64+\ep/16<\ep.
\eneq
By (\ref{Lher-1+}),
\beq\label{Lher-4}
p-pep\lesssim a.
\eneq
Since $\|(p-pep)-(p-e_0)\|<\ep/{32},$ by {Proposition 2.2 and Lemma 2.3 (b) of \cite{Rordam2}},
\beq\label{Lher-5}
p-e_0=f_{\ep/{16}}(p-e_0)\lesssim p-pep\lesssim a.
\eneq
We also estimate that  
\beq\label{Lher-5}
{\rm dist}(e_0xe_0, C_1)<\ep\rforal x\in {\cal F}.
\eneq
It follows that $pAp\in TA{\cal C}.$
\end{proof}

\begin{thm}\label{T1}
 Let $A$ be a unital simple separable \CA.
 Then $A\otimes C\in TA{\cal C}$ for all unital simple AF-algebra \textcolor{blue}{$C$} if and only if 
 $A\otimes C\in TA{\cal C}$ for \textcolor{blue}{some} infinite simple AF-algebra $C.$ 

\end{thm}

\begin{proof}
\textcolor{blue}{Following Theorem \ref{TN1}}, it suffices to show the following:
Suppose that $A\otimes C\in TA{\cal C}$ for some unital simple AF-algebra {$C$}. Then $A\otimes Q\in TA{\cal C}.$ 

Since every finite dimensional \CA\, is in ${\cal C},$ it is easy to see that 
we only need to consider the case that $A$ is infinite dimensional. 


 Let  $B=A\otimes Q.$
Let $1/4>\ep>0$. {Let} $a\in B_+\setminus \{0\}$ and let  ${\cal F}\subset B$ be a finite subset.
To simplify the notation, without loss of generality, we may assume 
that $\|x\|\le 1$ for all $x\in {\cal F}$ and $\|a\|=1.$ 


We will write $A\otimes Q$ as $\lim_{k\to \infty} (A\otimes M_{k!}, j_k),$ where
$j_k \colon A\otimes M_{k!}\to A\otimes M_{(k+1)!}$ {is given} by
$j_k(a)=a\otimes 1_{M_{(k+1)}}$ for all $a\in A\otimes M_{k!},$ { $k=1,2,....$ }   Without loss of generality, we may assume
that ${\cal F}\subset A\otimes M_{k!}$ for some $k\ge 1.$

Without loss of generality {again}, we may assume that there exists a positive element $a' \in A \otimes M_{k!}$ such that $\| a - a' \| < \ep$. 
By Proposition 2.2 and Lemma 2.3 (b) of \cite{Rordam2}, $f_{\ep}(a') \lesssim a$.  {Put $a_0=f_{\ep}(a')$. }
As $\| a \| = 1$ {and} $\ep < 1/4$, it is clear that $a_0 \in (A\otimes M_{k!})_+ \setminus \{0\}$.

For $C$ {in the statement}, we write it as $\lim_{m\to\infty} (C_m, \imath_m),$  { where {each} $C_m$
is a finite dimensional \CA\, and} $\imath_m$  is a unital
embedding {of $C_m$ into $C_{m + 1}$}.  Since $C$ is an infinite dimensional unital simple AF-algebra, {for  $k$ above,} we {can assume that for $m$ large enough, each
$C_m$ satisfies}
\beq\label{T1-1}
C_m=M_{n_1}\oplus M_{n_2}\oplus\cdots \oplus M_{n_{s(m)}},
\eneq
where $n_j\ge k!,$ $j=1,2,...,s(m)$.
Fix one such $m.$
{Then} one obtains a projection $q\in C_m$ {such} that $M_{k!}$ is a unital
\SCA\, of $qC_mq$ (with unit $p$). Put $e={1_A} \otimes q$ {in $A \otimes C$} and let {$\phi_1': M_{k!}\to qC_mq$ be} a unital embedding.
Define $\phi_1: A\otimes M_{k!}\to A\otimes qC_mq$ by
$\phi_1({x} \otimes {y}) = {x} \otimes \phi_1'({y})$ for all ${x} \in A$ and ${y} \in M_{k!}.$

By \textcolor{blue}{Proposition} \ref{Lher}, $e(A\otimes C)e\in TA{\cal C}.$ 
Therefore there exists a projection $p\in e(A\otimes C)e$ and a \SCA\, $I_0 \in {\cal C}$
of $e(A\otimes C)e$ with $1_{I_0}=p,$ {satisfying}
\beq\label{T1-2}
\|px-xp\| &<&\ep/16\tforal x\in \phi_1({\cal F}),\\\label{T1-2+}
{\rm dist}(pxp, I_0)&<&\ep/16\tforal x\in \phi_1({\cal F})\andeqn\\\label{T1-2++}
1-p &\lesssim &\phi_1(a_0).
\eneq

Choose a finite set $\mathcal{G}_0$ in $I_0$
such that $p {\varphi_1(\mathcal{F})} p \subset_{\ep / {16}} \mathcal{G}_0$. 
Since $I_0$ is weakly semi-projective, 
for $n$ large enough, there exists a
homomorphism $h \colon I_0 \rightarrow A \otimes (q C_n q)$ such that $\| h(y) - y \| < \ep / 32$ for all $y \in {\mathcal{G}_0}$.
Without loss of generality, replacing $h$ by ${\rm Ad}\, u\circ h$ for some unitary $u$ if necessary, we may assume 
that $h(p)=p.$ 
Using property (5), we obtains a unital \SCA\, $I_{00}\subset h(I_0)$ with $1_{I_{00}}=1_{h(I_0)}=p$ such that
$I_{00}\in {\cal C},$ 
\beq\label{T1-15-1}
{\rm dist}( g, I_{00})<\ep/16\rforal g\in {\cal G}_0.
\eneq
Therefore 
\beq\label{T1-2.1}
{\rm dist}(pxp, I_{00}) &<& \ep/4 \tforal x\in \phi_1({\cal F})
\eneq

Write $qC_nq$ as $M_{m_1}\oplus M_{m_2}\oplus \cdots M_{m_r}.$ Note
 that $k! | m_j$ for $j=1,2,...,r{,}$ as $\varphi_1'$ is unital. Put { $N=\sum_{j=1}^r m_j.$}
{Then} there is a unital  embedding $\phi_2': qC_nq\to M_{N!}.$
Consider \textcolor{blue}{the canonical embedding} $j_k: M_{k!}\to M_{N!}$ and $\phi_2'\circ \phi_1': M_{k!}\to M_{N!}.$
Since they {are both} unital, there is a unitary $u\in M_{N!}$ such that
$$
{\rm Ad} u \circ \phi_2'\circ \phi_1'=j_k.
$$

Define $\phi_2: A\otimes qC_nq\to A\otimes M_{N!}$ by
$$
\phi_2({x} \otimes {y}) = {x} \otimes ( {\rm Ad}\, u\circ \phi_2'({y}) )
$$
for all ${x} \in A$ and ${y} \in qC_nq.$

Then
\beq\label{T1-3}
(\phi_2\circ \phi_1) (c)=c\tforal c\in A\otimes M_{k!}.
\eneq

Put $p_1=\phi_2(p)\in A\otimes M_{N!}\subset A\otimes Q$
and $D=\phi_2(I_{00})\subset A\otimes M_{N!}\subset A\otimes Q$ with
$1_D=p_1.$ Note also $D\in {\cal C}.$ Moreover,
by (\ref{T1-2}), (\ref{T1-2+}) and (\ref{T1-2.1}), we have
\beq\label{T1-4}
&&\hspace{-0.4in}\|p_1x-xp_1\|=\|\phi_2(p\phi_1(x)-\phi_1(x)p)\|=\|p\phi_1(x)-\phi_1(x)p\|<\ep/2\tforal x\in {\cal F};\\
&&{\rm dist}(p_1xp_1, D)\le {\rm dist}(p\phi_1(x)p, I_{00})<\ep/2\tforal x\in {\cal F}.
\eneq
{Then},  by (\ref{T1-2++}),
\beq\label{T1-5}
1-p_1 = \phi_2(1-p)\lesssim \phi_2(\phi_1(a)) = a.
\eneq

This implies that $A\otimes Q\in TA{\cal C}.$

\end{proof}

The following corollary is a special case of Theorem \ref{T1}.

\begin{cor}\label{CN1}
Let $A$ be a unital simple separable \CA, and let $C$ be a unital infinite dimensional simple {AF-algebra}.  Suppose that $A\otimes C$ has tracial rank at most one.
Then  $A\in {\mathcal A}_1.$
\end{cor}

\section{Criterions for \CA s to be in ${\cal A}_1$}


\begin{lem}\label{ML1}
Let $A$ be a unital separable simple \CA{}{. Let} $C$ be a unital simple AH-algebra with no dimension growth and with ${\rm Tor}(K_0(C))=\{0\}$. Suppose that
{$A \otimes C$ is in $TA{\cal C}$. Then for any simple unital infinite dimensional AF algebra $F$, $A \otimes F$ is also in $TA{\cal C}$.}
\end{lem}

\begin{proof}

{According to Theorem \ref{T1}, we just need to show that $A \otimes F$ is in $TA{\cal C}$ for some simple unital AF-algebra \textcolor{blue}{$F$}.}

{ By Theorem 3.7 of \cite{Martin-Pasnicu}, we know   that $K_0(C)$ is weakly unperforated. By Theorem 2.7 of \cite{Goodearl-RieszDecomp}, $K_0(C)$ has the Riesz interpolation property.
As $\mbox{Tor}(K_0(C)) = 0$,
we have that $K_0(C)$ is an unperforated Riesz group.}
It follows from the Effros-Handelman-Shen theorem (Theorem 2.2 of \cite{EHS}) that there exists a unital separable simple AF-algebra $B$ with
\begin{align}\label{ML1-1}
(K_0(B), K_0(B)_+,[1_B])=(K_0(C), K_0(C)_+, [1_C]).
\end{align}

{We will show that $A\otimes B$ is in $TA{\cal C}.$}
For that, let $1/4>\ep>0$. {Let} ${\mathcal F}\subset A\otimes B$ be a finite subset and
let $a\in (A\otimes B)_+\setminus \{0\}$.  Without loss of generality, we may assume that $1/2>\ep,$ ${\mathcal F}$ is a subset of the unit ball and $\|a\|=1.$

{}{For any $f \in {\mathcal F}$, we} may assume that there are $a_{f,1},a_{f,2},...{,} a_{f,n(f)}\in A$
and $b_{f,1},b_{f,2},...,b_{f, n(f)}\in B$ such that
\begin{align}\label{ML1-1+1}
\|f-\sum_{{i}=1}^{n(f)} a_{f, {i}}\otimes b_{f,i} \|<\ep/32 .
\end{align}
We may also assume that there exist $x_1,x_2,...,x_{n(a)}\in A$ and
$y_1,y_2,...,y_{{n(a)}}\in B$ such that
\begin{align}\label{ML1-1+2}
\|f_{1/4}(a)-\sum_{i=1}^{{n(a)}} x_{{i}}\otimes y_i\|<\ep/32 .
\end{align}

Let
\begin{align}\label{ML1-1+3}
K_1&= { n(a) +\max\{ n(f) \colon f\in {\mathcal F}\} },\\
K_2 &= { \max\{\|x_i\|+\|y_i\| \colon 1\le i\le n(a)\} } \andeqn\\
K_3&= \max\{\|a_{f,{i}}\|+\|b_{f,i}\| \colon { 1\le i\le n(f) } \andeqn
f\in {\mathcal F}\}.
\end{align}

Put $a_1=f_{1/2}(a).$ 

As $B$ is an AF-algebra and $C$ has stable rank one (see \cite{DNNP}),  it is  known 
that there exists a unital \hm\, ${\phi'} \colon B\to C$
such that $({\phi'})_*$ gives the identification \eqref{ML1-1}.
Define ${\phi} \colon A\otimes B\to A\otimes C$ {as}
${\phi} = {\rm id}_A\otimes {\phi'}.$
Now since $A\otimes C$ is in $TA{\cal C}$,  there exists
{a \SCA\, $D$ of $A\otimes C$ such that $D \in {\mathcal C}$ and {(using $p$ to denote $1_D$)}}
\begin{align}\label{ML1-2}
\|px-xp\| &< \ep/32\tforal x\in {\phi}({\mathcal F}),\\
{\rm dist}(pxp, D) &< \ep/32\tforal x\in {\phi}({\mathcal F})\andeqn\\\label{ML1-2++}
1-p &\lesssim  {\phi}(a_1).
\end{align}
\textcolor{blue}{Thus} there exists $w\in A\otimes C$ such that 
$w^*w=1-p$ and $ww^*\phi(f_{1/4}(a))=ww^*.$
Let ${\cal G}_0\subset D$ be a finite subset such that, for each $x\in {\cal F},$ there exists 
$y\in {\cal G}_0$ such that $\|x-y\|<\ep/32.$

By the UCT,  we obtain $\kappa\in KL(C, B)$ such that
$\kappa|_{K_1(C)}=0$ and $\kappa|_{K_0(C)}=({\phi'})_{*0}^{-1}.$
Choose a unital AH-algebra $C_0$ with no  dimension growth whose Elliott invariant is
$$
(K_0(C_0), (K_0(C_0))_+,[1_{C_0}], K_1(C_0),  T(C_0), r_{C_0})=(K_0(C), (K_0(C)_+, [1_{C}], \{0\}, T(C), r_C).
$$
With the identification above, it is known (by Theorem 6.10 of \cite{Lin-HomFromAH}, for example) that 
there exists a unital \hm\, $H: C\to C_0$ such that 
$H_{*0}={\rm id}_{K_0(C)},$ $H_{*1}=0$ and $H$ induces the identity map on $T(C).$ 
By the UCT, we obtain $\kappa\in KL(C_0, B)$ such that $\kappa|_{K_0(C_0)}={\rm id}_{K_0(C_0)}$ (with the above identification). Note $K_1(C_0)=\{0\}=K_1(B).$ 
It follows from  Theorem 9.12 of \cite{Lin-TR1} (see also Theorem 4.7 of \cite{Lin-KT} ) 
that there exists a  sequence of  unital 
\morp s $\Psi_n': C_0\to B$ such that
\beq\label{ML1-3}
[\{\Psi_n'\}]=\kappa\andeqn \lim_{n\to\infty}\|\Psi_n'(x)\Psi_n'(y)-\Psi_n'(xy)\|=0
\eneq
for all $x,\, y\in \textcolor{blue}{C_0}.$ 
Define $\Psi_n=\Psi_n'\circ H: C\to B,$ $n=1,2,....$ 
By {Theorem 3.5.3 of \cite{BO-Book}}, there exists a sequence of unital \morp s $\Phi_n: A\otimes C\to A\otimes B$ such that
\beq\label{ML1-3+1}
\Phi_n(x\otimes y)=x\otimes \Psi_n(y)
\eneq
for all $x\in A$ and $y\in C.$ 
Since $D$ is weakly semiprojective, we may assume that, there exists $n_0\ge 1$ such that, for all 
$n\ge n_0,$ there exists a unital \hm\, $h_n: D\to A\otimes B$ such that
\beq\label{ML1-3+2}
\|h_n(g)-\Phi_n(g)\|<\ep/32\rforal g\in {\cal G}_0.
\eneq

Let
\beq\label{ML1-4}
{\mathcal G}_1 & = & \{  \textcolor{blue}{1_B}  \}\cup \{ y_i \colon 1\le i\le {n}(a)\}\cup\{ b_{f,i} \colon { 1\le i\le n(f)} \andeqn f\in {\mathcal F} \}
\andeqn\\
{\cal G}_2 & = & \{a_{f,i}: 1\le  i\le n(f) \  \mbox{\textcolor{blue}{and}}  \  f\in {\cal F}\}\cup \{x_i: 1\le i \le n(a)\}.
\eneq
\textcolor{blue}{As $B$ is an AF algebra}, without loss if generality, we may assume that there exists a finite dimensional \SCA\, \textcolor{blue}{$E \subset B$}
such that ${\cal G}_1\subset \textcolor{blue}{E}.$

Put
\begin{align}\label{ML1-5}
\dt={\ep\over{32K_1K_2K_3}}.
\end{align}
\textcolor{blue}{As $E$ is weakly semiprojective, so is $\varphi'(E)$ (note that $E$ is simple)}. There \textcolor{blue}{then} exists a unital \hm\, $h_n':  \textcolor{blue}{\varphi'(E)} \to B$ such that, when $n$ is large enough, 
\beq\label{ML1-5+2}
\|h_n'(g)-\Psi_n(g)\|<\dt/2\rforal g\in \textcolor{blue}{ \varphi' ( {\cal G}_1) }.
\eneq

We may also assume, without loss of generality, that 
$(\textcolor{blue}{h_n' \circ \phi'})_{*0}=(({\rm id}_{B})|_{ \textcolor{blue}{E} })_{*0}.$ 
\textcolor{blue}{Then we can choose sufficiently large $n_1,$  such that for each $n > n_1$,} there exists a unitary $\textcolor{blue}{v_n} \in C$
satisfying
\begin{align}\label{ML1-6}
\|{ ( \rm Ad}\, \textcolor{blue}{v_n} \circ h_n' \circ {\phi'} ) (y)-y\|<\dt/2\tforal y\in {\mathcal G}_1.
\end{align}

\textcolor{blue}{For $n\ge n_1$,} define $\Phi_n' \colon A\otimes C\to A\otimes B$ by
$\Phi_n'={\rm Ad}\, (\textcolor{blue}{1_A} \otimes v) \circ \Phi_n.$
Put $p_1 = h_n'(p)$ and $D_1=h_n(D).$  
By choosing  even larger $n_1,$ we may assume, without loss of generality
\beq\label{ML1-6+}
\|\Phi_n'(w)^*\Phi_n'(w)-(1-p_1)\|<\ep/16\andeqn\\\label{ML1-6++}
 \|\Phi_n'(w)\Phi_n'(w)^*(\Phi_n'\circ \phi(f_{1/4}(a)))-\Phi_n'(w)\Phi_n'(w)^*\|<\ep/16.
\eneq

Then, one estimates, by \eqref{ML1-6} and \eqref{ML1-1+1}, that
\begin{align}\label{ML1-7}
\|\textcolor{blue}{\Phi_n'} \circ {\phi}(f)-f\|<\ep/32+\ep/32+K_1K_3\dt<3\ep/32\tforal f\in {\mathcal F}.
\end{align}
Similarly,
\begin{align}\label{ML1-8}
\|\textcolor{blue}{\Phi_n'} \circ {\phi}(f_{1/4}(a))-f_{1/4}(a)\|<\ep/32+\ep/32+K_1K_2\dt<3\ep/32 {}{.}
\end{align}
By applying \eqref{ML1-2}, \eqref{ML1-3+2} and \eqref{ML1-7}, we {}{then have} that
\begin{align}\label{ML1-9}
\|p_1x-xp_1\| &\le \|p_1x- \textcolor{blue}{\Phi_n'} (p)  \textcolor{blue}{\Phi_n'} \circ {\phi}(x)\|\\
& \hspace{.7cm} +\| \textcolor{blue}{\Phi_n'} (p) \textcolor{blue}{\Phi_n'} \circ {\phi}(x) - \textcolor{blue}{\Phi_n'} \circ {\phi}(x) \textcolor{blue}{\Phi_n'} (p)\|\\
& \hspace{.7cm} +\| \textcolor{blue}{\Phi_n'} \circ {\phi}(x) \textcolor{blue}{\Phi_n'} (p)-xp_1\|\\
&< 3\ep/16+\|p {\phi}(x)-{\phi}(x)p\|+3\ep/16 \\
& <7\ep/16 \label{ML1-9end}
\end{align}
for all $x\in {\mathcal F}.$
Similarly,
\begin{align}\label{ML1-10}
{\rm dist}(p_1xp_1, D_1 )< {\ep / {2}} \tforal x\in {\mathcal F}.
\end{align}
By property (5) of \CA s in ${\cal C},$ there is \textcolor{blue}{a} \SCA\, $D_2\subset D_1$ with $1_{D_2}=1_{D_1}=p_1$ such that
$D_2\in {\cal C}$ and 
\beq\label{ML1-10+1}
{\rm dist}(p_1xp_1, D_2)<\ep\rforal x\in {\cal F}.
\eneq

Now, by \eqref{ML1-6+} and \eqref{ML1-6++},
there are projections $e_1\in \overline{\Phi_n'(w)^*\Phi_n'(w)(A\otimes B)\Phi_n'(w)^*\Phi_n'(w)}$
and $e_2\in \overline{\Phi_n'(w)\Phi_n'(w)^*(A\otimes B)\Phi_n'(w)\Phi_n'(w)^*}$ 
such that $e_1\sim e_2$ and 
\beq\label{ML1-10+2}
\|e_1-\Phi_n'(w)^*\Phi_n'(w)\|<\ep/8\andeqn \|e_2-\Phi_n'(w)\Phi_n'(w)^*\|<\ep/8.
\eneq
Moreover,
\beq\label{ML1-10+3}
\|(1-p_1)-e_1\|<\ep/8\andeqn \|e_2\Phi_n'\circ \phi(f_{1/4}(a))-e_2\|<\ep/4.
\eneq
It follows from \eqref{ML1-8} that
\beq\label{ML1-10+4}
\|e_2f_{1/4}(a)-e_2\|<\ep/2+3\ep/32.
\eneq
Thus
\beq\label{ML1-10+5}
\|f_{1/4}(a)e_2f_{1/4}(a)-e_2\|<  \ep + 6 \ep/32  <1/2.
\eneq
\textcolor{blue}{Then we can find a projection in $\mbox{Her}(f_{1/4}(a))$ which is unitarily equivalent to $e_2$.}  It follows that $e_2\lesssim f_{1/16}(a).$  Therefore, by \eqref{ML1-10+3}, 
\beq\label{ML1-10+6}
1-p_1\sim e_1\sim e_2\lesssim f_{1/16}(a)\lesssim a.
\eneq
From \eqref{ML1-9end}, \eqref{ML1-10}, and \eqref{ML1-10+6}, we conclude that $A\otimes B$ is in $TA{\cal C}.$ 
By Theorem \ref{T1}, for any unital simple infinite dimensional AF algebra $F$, $A \otimes F \in TA{\mathcal C}$.

\end{proof}

\begin{thm}\label{T2}

Let $A$ be a unital separable simple \CA. Suppose that
$A\otimes C$ is in $TA{\cal C}$  for some
unital amenable separable simple \CA\, {$C$} such that $TR(C)\le 1$  and $C$ satisfies the UCT.
Then {$A \otimes F$ is in $TA{\mathcal C}$ for any simple unital infinite dimensional AF algebra \textcolor{blue}{$F$}.}

\end{thm}

\begin{proof}
We may assume that $C$ has infinite dimension. Otherwise, as $C$ is simple,
$C \cong M_n(\mathbb{C})$ for some $n \in \N$. {As $M_n(A)$ is in $TA{\cal C}$, 
by applying Proposition \ref{Lher}, 
we conclude that $A$ is also in $TA{\cal C}.$ It follows from  Proposition \ref{L1} that  $A \otimes F$ is in $TA{\cal C}$ for any unital simple infinite dimensional AF algebra $F$.}

Now assume that  $C$ is {infinite} dimensional. { It follows from the assumption that  $A \otimes C$ is in $TA{\cal C}$ and  from Proposition \ref{L1} that 
{$(A \otimes C) \otimes Q$ is in $TA{\cal C}$}.}
Note that {$(A \otimes C) \otimes Q \cong A \otimes (C\otimes Q)$}.  Since  $TR(C) \leq 1$, it follows that $TR(C\otimes Q)\le 1.$
Since $C$ is amenable and satisfies {}{the} UCT, $C\otimes Q$ is also a unital separable amenable simple \CA\, which satisfies the UCT.  {It follows
from Theorem 10.4 of \cite{Lin-TR1} that $C \otimes Q$ is a unital simple AH-algebra with no dimension growth.}
One computes that $K_0(C\otimes Q)$ is torsion free.  \textcolor{blue}{Applying Lemma \ref{ML1}, we have that} $A \otimes F$ is in $TA{\cal C}$ for any unital simple infinite dimensional AF algebra $F$.
\end{proof}

{As \textcolor{blue}{a special case} to Theorem \ref{T2}, we have the following corollary.}

\begin{cor}\label{MC3}
Let $A$ be a unital separable simple \CA. 
Then $A\in {\cal A}_0$ if and only if $TR(A\otimes C)=0$ for some unital amenable separable simple \CA\,
$C$ with $TR(C)\le 1$ which satisfies the UCT.
\end{cor}


\begin{cor}\label{MC4}
Let $A$ be a unital separable simple \CA. 
Then $A\in {\cal A}_1$ if and only if $TR(A\otimes C)\le 1$ for some 
unital  simple $AH$-algebra $C.$

\end{cor}

\begin{proof}
For the ``only if "part, we only need to choose $C$ to be $Q.$ The corollary then follows from Corollary \ref{CN1}.

For the ``if" part, note that  by Theorem 10.4 of \cite{Lin-TR1} {,} $C\otimes Q$ is a unital simple AH-algebra with no dimension growth. {Since $TR(A \otimes C) \leq 1$,} {}{we have} $TR(A\otimes C\otimes Q)\le 1.$ Theorem \ref{T2} then applies.
\end{proof}

%
%
%
%

%
%
%
%
%
%
%
%
%
%



%
%
%
%

\section{Tensor Products}
\setcounter{equation}{0}

In this section we are ready to answer the following three questions:

(1) Let $A$ and $B$ be both in ${\cal A}_1\cap {\cal N}.$ Is $A\otimes B$ in  ${\cal A}_1\cap {\cal N}?$
 
 (2) Let $A$ be a unital separable simple \CA\, and let $B\in {\cal A}_1\cap {\cal N}.$ 
 Suppose that $A\otimes B\in {\cal A}_1.$ Is it true that $A\in {\cal A}_1?$ 
 
(3) Let $A\in {\cal A}_1\cap {\cal N}$ and $B\in {\cal N}$ with $TR(B)\le 1.$ 
Is it true $TR(A\otimes B)\le 1?$

\begin{prop}\label{T5}
Let $A$ {and} $B$ be two unital separable simple \CA s in {${\mathcal A}_1 \cap {\mathcal N}.$}
Then {$A\otimes B \in {{\mathcal A}_1 \cap {\mathcal N}}.$}

\end{prop}

\begin{proof}
Let $A, B \in {{\mathcal A}_1 \cap {\mathcal N}}$. Then
$$(A\otimes B)\otimes Q \, {\cong} \, (A\otimes B) \otimes (Q\otimes Q) \, {\cong} \, (A\otimes Q) \otimes (B\otimes Q).$$

Since both $A$ and $B$ are in  {${\mathcal A}_1 \cap {\mathcal N}$}, $A\otimes Q$ and $B\otimes Q$ {have tracial rank no more than one and} satisfy the UCT.
Therefore,  by Lemma 10.9 and Theorem 10.10 of \cite{Lin-TR1}, each of them is isomorphic to some unital simple AH-{{algebra}} with no dimension growth.
It is then easy to see that $(A\otimes Q)\otimes (B\otimes Q)$ can be written as a unital simple AH-algebra
with no dimension growth, {which implies that $TR(A \otimes B \otimes Q) \leq 1$.}

\end{proof}

\begin{thm}\label{T6}
Let $A$ be  a unital separable simple \CA.  Suppose that there exists a unital separable simple  \CA\, $B\in  {{\mathcal A}_1 \cap {\mathcal N}}$ such that
$A\otimes B\in {\mathcal A}_1,$ then $A \in {{\mathcal A}_1}$.

\end{thm}

\begin{proof}
{
Since $A\otimes B\in {\mathcal A}_1$, {}{we have that} $TR(A\otimes B\otimes Q)\le 1.$
As $B\in {\mathcal A}_1 {\cap {\mathcal N}}$, we have that $B \otimes Q$ satisfies {}{the} UCT and $TR(B \otimes Q) \leq 1$. By Lemma 10.9 {and} Theorem 10.10 of \cite{Lin-TR1}{,} $B\otimes Q$ is a unital simple AH-algebra with no dimension growth.
Note that ${\mbox{Tor}}(K_0( B \otimes Q )) = 0$.
It follows from Lemma \ref{ML1} (by setting $TA{\cal C}$ to $TAI$ algebras) that  $A\in {\mathcal A}_1.$}

\end{proof}

We now consider the converse of a special case of Theorem \ref{T2} (when $TA{\cal C}$ are just $TAI$ algebras) in the following sense.
Let $A\in {\mathcal A}_1 {\cap {\mathcal N}}$. Is it true that ${TR}(A\otimes C)\le 1$ if $C$ is a unital separable
infinite dimensional simple \CA\, with $TR(C)\le 1$? {}{An affirmative answer is given in Theorem \ref{FT}}.

\begin{lem}\label{RriesztoR}
Let $G$ be a countable weakly unperforated  simple ordered group which is rationally
Riesz.  Suppose that $G$ also has the following property:
for any $x,\, y\in G$ with $x<y$ and for any integer $N \ge 1$,
there exists $z\in G$ such that
\beq\label{RtoR-1}
x< Nz <  y.
\eneq
Then $G$ has the Riesz interpolation property.
\end{lem}

\begin{proof}
Let $u\in G_+$ be an order unit. Denote by $S_u(G)$ the state space of $G$, i.e.,  the set of order and unit preserving homomorphisms from $G$ to the additive group $\R$.
First, we claim the following:
For any $a_1, a_2\in G_+\setminus \{0\},$ there is $b\in G_+\setminus\{0\}$ such that
\beq\label{RtoR-1+}
0<b<a_i,\,\,\,i=1,2.
\eneq
In fact, as $G$ is simple, there exists an integer $n_1>0$ such that
\beq\label{RtoR-1+2}
n_1a_i>u, \  i=1,2.
\eneq
By the assumption, there exists $b_0\in G$ such that
\beq\label{RtoR-1+3}
0<n_1b_0<u.
\eneq
As $G$ is weakly unperforated, we get
\beq\label{RtoR-1+4}
0<b_0<a_i,\,\,\,i=1,2,
\eneq
which proves the claim.

Suppose that $x_i\le y_j$ for $i, j=1,2.$
We will show that there exists $z\in G$ such that
\beq\label{RtRn}
x_i\le z\le y_j,\,\,\, i,j=1,2.
\eneq

If $x_{i'}=y_{j'}$ for some
pair of $i'$ and $j',$ choose $z=y_{j'}.$ Then
$x_i\le y_{j'}=z=x_{i'}\le y_i,$ $i=1,2.$
Now assume that $x_i<y_j$ for all $i$ and $j.$

Since $G$ is rationally Riesz, there are $m, n\in \N\setminus \{0\}$  and
$w\in G$ such that

\beq\label{RtoR-2}
nw \le my_j\andeqn mx_i\le nw, \,\,\, i,j=1,2.
\eneq

If $nw=mx_{i'}=my_{j'}$ for certain $i'$ and $j',$
then $m(y_{j'}-x_{i'})=0.$  Since $y_{j'}-x_{i'}>0$ and  $G$ is an ordered group, this is impossible.

If $nw<my_j$ for all $j,$  by the claim above,
there exists $b_0\in G_+$ such that
\beq\label{RtoR-2+}
0<b_0<my_j-nw,\,\,\, j=1,2.
\eneq

By the assumption, there exists
$z\in G$ such that
\beq\label{RtoR-3}
mx_i\le nw<mz<nw+b_0<my_j,\,\,\, j=1,2.
\eneq
By the weak unperforation,
\beq\label{RtoR-4}
x_i<z<y_j,\,\,\, i,j=1,2.
\eneq

If $nw>mx_i,$ $i=1,2,$ by the claim, there exits $b_0\in G_+$ such that
\beq\label{RtoR-5}
0<b_0<nw-mx_i,\,\,\, i=1,2.
\eneq
Then, as above,  we obtain $z\in G$ such that
\beq    \label{RtoR-6}
mx_i< nw+b_0<mz<nw\le my_j,\,\,\,i.j=1,2.
\eneq
We then conclude, as above,
\beq\label{RtoR-7}
x_i<z < y_j,\,\,\, i,j=1,2.
\eneq

Thus $G$ has the Riesz interpolation property.

\end{proof}

\begin{lem}\label{ABRiesz}
Let $A\in {{\mathcal A}_1} {\cap {\mathcal N}}$. Suppose that $B$ is a unital separable amenable simple \CA\, with $TR(B)\le 1$ which satisfies the UCT.
Then $K_0(A\otimes B)$ has the {Riesz} interpolation property.
\end{lem}

\begin{proof}

Since $A \in {\mathcal A}_1 {\cap {\mathcal N}}$ and $TR(B) \le 1$,  by {}{Proposition} \ref{T5}, $A \otimes B \in {\mathcal A}_1 {\cap {\mathcal N}}$.
It follows from \cite{LN-range} that $K_0(A \otimes B)$ is rationally {Riesz}.

By Lemma 10.9 and      Theorem 10.10 of \cite{Lin-TR1},
 $B$ is isomorphic to a unital simple AH-algebra with no dimension growth. It follows from Theorem 2.1 of \cite{EGL2} that
$B$ is approximately divisible. Therefore $A\otimes B$ is approximately divisible. It follows that, for any pair $x, y\in K_0(A\otimes B)$ and any
integer $N\ge 1$ with $x<y,$ there exists $z\in K_0(A\otimes B)$ such that
\beq\label{ABR-1}
x<Nz<y.
\eneq
Moreover, \textcolor{blue}{from the approximate divisibility}, by Theorem 1.4 of \cite{BKR}, $A \otimes B$ has \textcolor{blue}{the} strict comparison for positive elements. In particular, it follows that $K_0(A\otimes B)$ is weakly unperforated.
The lemma then follows by applying Lemma \ref{RriesztoR}.

\end{proof}

\begin{thm}\label{ABtr1}
Let $A\in {\mathcal A}_1 {\cap {\mathcal N}}$. Then, for any unital infinite dimensional simple AH-algebra $B$ with slow dimension growth,
$A\otimes B$ is a unital simple AH-algebra with no dimension growth.
\end{thm}

\begin{proof}
Since $A\in {\mathcal A}_1 {\cap {\mathcal N}}$, it follows {from} {}{Proposition} \ref{T5} that $A\otimes B\in {\mathcal A}_1 {\cap {\mathcal N}}$. By {Lemma} \ref{ABRiesz},
$K_0(A\otimes B)$ has the Riesz interpolation property. Since $B$ is {an} infinite dimensional simple AH-algebra of no dimension growth,  \textcolor{blue}{from Theorem 2.1 of \cite{EGL2}}, $B$ is approximately divisible. So
  $A\otimes B$ is approximately divisible. It follows that
$K_0(A\otimes B)\not = \Z$.
Since $A\otimes B\otimes Q$ is a unital simple AH-algebra
of no dimension growth, the  canonical map $r_{A\otimes B\otimes Q}\colon T(A\otimes B\otimes Q)\rightarrow S_{[1]}(K_0(A \otimes B\otimes Q))$ maps extreme points to extreme points. Therefore the
 canonical map
$r_{A\otimes B} \colon T(A \otimes B) \rightarrow S_{[1]}(K_0(A \otimes B))$
maps the {extreme} points to 
{{extreme}} points (see lemma 5.6 of \cite{LN-range}). 
It follows from \cite{Villadsen1} that there is a unital simple AH-algebra {$C$} with no dimension growth
such that {}{its} Elliott invariant is exactly the same as that of $A\otimes B.$
According to Theorem 10.4 of \cite{Lin-TR1}, we have that $A \otimes B \cong C.$

\end{proof}

{We end {}{this} note by the following summarization:}

\begin{thm}\label{FT}

Let $A\in {\mathcal N}$ be a unital separable simple amenable \CA \ {that {satisfies} the UCT}.
Then the following are equivalent.

(1) $A\in {\mathcal A}_1$;

(2) $TR(A\otimes Q)\le 1;$

(3) $A\otimes Q\in {\mathcal A}_1;$

(4) $TR(A\otimes B)\le 1$ for some unital infinite dimensional simple AF-algebra $B;$

(5) $TR(A\otimes B)\le 1$ for all unital simple {infinite dimensional} AF-algebras $B;$

(6) $A\otimes B\in {\mathcal A}_1$ for some unital simple infinite dimensional AF-algebra $B;$

(7) $A\otimes B\in {\mathcal A}_1$ for all unital simple {infinite dimensional} AF-algebras $B;$

(8) $TR(A\otimes B)\le 1$ for some unital infinite dimensional simple AH-algebra $B$ with no dimension growth;

(9) $TR(A\otimes B)\le 1$ for all unital simple {infinite dimensional} AH-algebras $B$ with no dimension growth;

(10) $A\otimes B\in {\mathcal A}_1$ for some unital simple {infinite dimensional} AH-algebra $B$ with no dimension growth;

(11) $A\otimes B\in {\mathcal A}_1$ for all unital simple {infinite dimensional} AH-algebras $B$ {with no dimension growth};

(12) $A\otimes B\in {\mathcal A}_1$ for {some} unital simple {infinite dimensional} \CA\, $B$ {}{in} ${\mathcal A}_1 {\cap {\mathcal N}}$;

(13) $A\otimes B\in {\mathcal A}_1$ for {all} unital simple {infinite dimensional} \CA {s}\, $B$ in ${\mathcal A}_1 {\cap {\mathcal N}}$.

\end{thm}

\begin{proof}

Note that  ``(1) $\Rightarrow$ (2)", ``(2) $\Rightarrow$ (3)",
``(5) $\Rightarrow$ (4)'',  ``(4) $\Rightarrow$ (6)", ``(7) $\Rightarrow$ (6)'',  ``(9) $\Rightarrow$ (8)'', ``(9) $\Rightarrow$ (10)'',  ``(11) $\Rightarrow$ (10)", ``(11) $\Rightarrow$ (7)'',
``({13}) $\Rightarrow$ (11)", ``({13}) $\Rightarrow$ (7)'' and ``({13}) $\Rightarrow$ ({12})'' are {}{straightforward}.

\vspace{2mm}

{}{Note} that ``(1) $\Rightarrow$  (5)"  and ``(1) $\Rightarrow$ (9)" follow from {}{Theorem} \ref{ABtr1}.
To see that ``(1) $\Rightarrow$ ({13})"{}{,}
let $A \in {{\mathcal A}_1 \cap {\mathcal N}}$ and $B\in {{\mathcal A}_1 \cap {\mathcal N}}$. Then
$TR(B\otimes Q)\le 1.$ So $B\otimes Q$ is a unital simple infinite dimensional AH-algebra with no dimension growth.
Since ``(1) $\Rightarrow$ (9)", this implies that $TR(A\otimes (B\otimes Q))\le 1.$
It {}{then} follows that $A\otimes B\in {\mathcal A}_1.$

 For ``({12}) $\Rightarrow$ (1)", {}{assume that} $TR(A\otimes B\otimes Q)\le 1.$ It follows
 that $TR(A\otimes (B\otimes Q))\le 1.$ Since $TR(B\otimes Q)\le 1,$ again, $B\otimes Q$ is a unital simple
 infinite dimensional AH-algebra with no dimension growth. It follows from Corollary \ref{MC4} that {$A\in {\mathcal A}_1$}.

That ``(3) $\Rightarrow$ (1)" follows from {\cite{LN-range}} and ``(4) $\Rightarrow$ (1){"} follows from Corollary \ref{CN1}.

For ``(6) $\Rightarrow$ (4)", one considers $A\otimes B\otimes Q$ and notes {}{that} $B\otimes Q$ is a unital simple
infinite dimensional AF-algebra.

That ``(8) $\Rightarrow$  (4)"  follows from {}{Corollary} \ref{MC4}.

The rest of implications follow similarly as established {}{above}.

\end{proof}

\vspace{1cm}

e-mail address: hlin@uoregon.edu, \,\,\,wsun@math.ecnu.edu.cn

\end{document}